\documentclass[12pt]{article}
\usepackage[a4paper]{geometry}
\usepackage{amssymb,amsmath,amsthm}
\usepackage{graphicx,cite,xcolor}
\usepackage{longtable,pdflscape,booktabs,caption,multicol}
\usepackage{multirow}
\usepackage[colorlinks=true,citecolor=black,linkcolor=black,urlcolor=blue]{hyperref}

\theoremstyle{plain}
\newtheorem{theorem}{Theorem}
\newtheorem{lemma}[theorem]{Lemma}
\newtheorem{proposition}[theorem]{Proposition}
\newtheorem{corollary}[theorem]{Corollary}

\newcommand{\arxiv}[1]{\href{https://arxiv.org/abs/#1}{\texttt{arXiv:#1}}}

\title{On the nullities of quartic circulant graphs\\ and their extremal null spaces}

\author{Ivan Damnjanovi\'c\thanks{The author is supported by Diffine LLC.}\\
\small University of Ni\v s, Faculty of Electronic Engineering,\\[-0.4ex]
\small Aleksandra Medvedeva 14, 18106 Ni\v s, Serbia\\[-0.4ex]
\small\tt ivan.damnjanovic@elfak.ni.ac.rs\\
\small Diffine LLC\\[-0.4ex]
\small 3681 Villa Terrace, San Diego, CA 92104, USA\\[-0.4ex]
\small\tt ivan@diffine.com}

\begin{document}

\maketitle

\begin{abstract}
A circulant graph is a simple graph whose adjacency matrix can be represented in the form of a circulant matrix, while a nut graph is considered to be a graph whose null space is spanned by a single full vector. In a previous study by Damnjanovi\'c [\arxiv{2212.03026}, 2022], the complete set of all the pairs $(n, d)$ for which there exists a $d$-regular circulant nut graph of order $n$ has been determined. Motivated by the said results, we put our focus on the quartic circulant graphs and derive an explicit formula for computing their nullities. Furthermore, we implement the aforementioned formula in order to obtain a method for inspecting the singularity of a particular quartic circulant graph and find the concise criteria to be used for testing whether such a graph is a nut graph. Subsequently, we compute the minimum and maximum nullity that a quartic circulant graph of a fixed order $n$ can attain, for each viable order $n \ge 5$. Finally, we determine all the graphs attaining these nullities and then provide a full characterization of all of their corresponding extremal null spaces.

\end{abstract}

\section{Introduction}

In this paper we will consider all graphs to be undirected, finite, simple and non-null. Thus, every graph will have at least one vertex and there shall be no loops or multiple edges. As usually done so in spectral graph theory, when talking about the spectral properties of a graph, we shall exclusively refer to the corresponding spectral properties of its adjacency matrix. Also, we shall use $\mathcal{N}(G)$ to denote the null space of the graph $G$ and $\eta(G)$ to signify its nullity. Finally, for convenience, we will take that each graph of order $n$ has the vertex set $\{0, 1, 2, \ldots, n-1\}$.

A singular graph is said to be a graph whose nullity is positive and a core graph represents a singular graph whose null space contains a full vector, i.e.\ a vector without zero elements. Furthermore, a special type of a core graph is a nut graph, which is defined as a core graph all of whose non-zero null space vectors are full, or, alternatively, as a core graph of nullity one. The chemical justification for studying such graphs is disclosed in various papers --- see, for example, \cite{Fowler1, Coolsaet, Fowler2}.

Now, we will consider the graph $G$ to be a circulant graph if its adjacency matrix $A_G$ has the form
\[
A_G = \begin{bmatrix}
    a_0 & a_1 & a_2 & \cdots & a_{n-1}\\
    a_{n-1} & a_0 & a_1 & \cdots & a_{n-2}\\
    a_{n-2} & a_{n-1} & a_0 & \cdots & a_{n-3}\\
    \vdots & \vdots & \vdots & \ddots & \vdots\\
    a_1 & a_2 & a_3 & \dots & a_0
\end{bmatrix}.
\]
Here, we clearly have $a_0 = 0$, as well as $a_j = a_{n-j}$ for all the $j = \overline{1, n-1}$. A concise way of describing a circulant graph is by taking into consideration the set of all the values $1 \le j \le \frac{n}{2}$ for which $a_j = a_{n-j} = 1$. We shall refer to this set as the generator set of a circulant graph and we will use $\mathrm{Circ}(n, S)$ to denote the circulant graph of order $n$ whose generator set is $S$.

The circulant graphs possess a key property that they are core if and only if they are singular, as we shall soon demonstrate. On top of that, these graphs are nut if and only if their nullity is exactly one, as demonstrated by Damnjanovi\'c and Stevanovi\'c \cite[Remark~5]{Damnjanovic1}. For these reasons, it makes sense to perform a more detailed study on the topic of circulant graphs. In fact, all the pairs $(n, d)$ for which there exists a $d$-regular circulant nut graph of order $n$ have been fully determined by Damnjanovi\'c \cite[Theorem 5]{Damnjanovic3}.

In a recent private communication, Toma\v{z} Pisanski suggested that it would be useful to obtain a full characterization of all the circulant nut graphs. Motivated by this comment, the primary goal of the given paper is to make a detailed analysis of the null spaces and nullities of all the quartic circulant graph. An additional purpose of this study is to obtain results that could provide insight and open the door to possible further conclusions regarding the circulant graphs of higher degrees as well.

The paper will start off by quickly showing the simple fact that each singular circulant graph must be a core graph. Afterwards, it will focus on providing an answer to the following four questions:
\begin{itemize}
    \item What is the nullity of a given arbitrarily chosen quartic circulant graph?
    \item When is a quartic circulant graph singular? When is it nut?
    \item For a given order $n \ge 5$, what is the minimum nullity that a quartic circulant graph of order $n$ can have? Which graphs attain the said nullity and what are their corresponding null spaces? What happens when we restrict ourselves only to connected graphs?
    \item For a given order $n \ge 5$, what is the maximum nullity that a quartic circulant graph of order $n$ can have? Which graphs attain the said nullity and what are their corresponding null spaces? What happens when we restrict ourselves only to connected graphs?
\end{itemize}

It is straightforward to notice that any quartic circulant graph must be representable as $\mathrm{Circ}(n, \{p, q\})$ for some $n \ge 5$ and $1 \le p < q < \frac{n}{2}$. Bearing this in mind, the remainder of the paper will be organized as follows. Section \ref{sc_preliminaries} shall serve to preview certain theoretical facts regarding the circulant graphs that will be used later on while proving the main results. In Section~\ref{sc_the_formula} we will derive an explicit formula for $\eta(\mathrm{Circ}(n, \{p, q\}))$ in terms of $n, p, q$ and then find the precise conditions that these parameters need to satisfy in order for $\mathrm{Circ}(n, \{p, q\}))$ to be singular and to be nut. Subsequently, we shall use Section \ref{sc_nullity_min} to deal with the aforementioned minimum nullity problem, both for all the quartic circulant graphs and for the connected ones only. Finally, we will provide a full solution to the maximum nullity problem in Section \ref{sc_nullity_max}.

\section{Preliminaries}\label{sc_preliminaries}

It is known from elementary linear algebra theory (see, for example, \cite[Section~3.1]{Gray}) that the circulant matrix
\[
A = \begin{bmatrix}
    a_0 & a_1 & a_2 & \cdots & a_{n-1}\\
    a_{n-1} & a_0 & a_1 & \cdots & a_{n-2}\\
    a_{n-2} & a_{n-1} & a_0 & \cdots & a_{n-3}\\
    \vdots & \vdots & \vdots & \ddots & \vdots\\
    a_1 & a_2 & a_3 & \dots & a_0
\end{bmatrix}
\]
must have the complex eigenvectors $\begin{bmatrix} 1 & \zeta & \zeta^2 & \cdots & \zeta^{n-1} \end{bmatrix}^\mathrm{T}$ corresponding to the eigenvalues $P(\zeta)$, respectively, as $\zeta$ ranges over the $n$-th roots of unity, where
\begin{equation}\label{general_circulant_p}
    P(x) = a_0 + a_1 x + a_2 x^2 + \cdots + a_{n-1} x^{n-1}.
\end{equation}
From here, it becomes easy to obtain the next proposition regarding the core property of circulant graphs.
\begin{proposition}\label{core_property}
    A circulant graph is core if and only if it is singular.
\end{proposition}
\begin{proof}
    Let $G$ be a given circulant graph. It is clear that if this graph is core, then it must be singular, hence it suffices to only prove the converse. Now, if we suppose that $G$ is singular, this means that there surely exists an $n$-th root of unity $\zeta$ such that $P(\zeta) = 0$, which further implies that $\mathcal{N}(G)$ certainly contains the full complex vector $u = \begin{bmatrix} 1 & \zeta & \zeta^2 & \cdots & \zeta^{n-1} \end{bmatrix}^\mathrm{T}$. It immediately follows that $\mathcal{N}(G)$ must also contain the real vectors $\mathrm{Re}(u)$ and $\mathrm{Im}(u)$. By denoting $u_\alpha = \mathrm{Re}(u) + \alpha \, \mathrm{Im}(u)$ for any $\alpha \in \mathbb{R}$, it is easy to establish that there exists an $\alpha$ for which $u_\alpha$ is a full real vector. This is a direct consequence of the fact that, for any $j = \overline{0, n-1}$, the equation $\mathrm{Re}(u_j) + \alpha \, \mathrm{Im}(u_j) = 0$ in $\alpha$ has at most one solution, by virtue of $u$ being a full complex vector. The proposition statement follows swiftly from here.
\end{proof}

In a similar manner, it is possible to obtain a connection between circulant graphs and nut graphs by implementing the following brief proposition that was previously shown by Damnjanovi\'c and Stevanovi\'c \cite[Remark 5]{Damnjanovic1}, as noted earlier.
\begin{proposition}\label{nut_property}
    A circulant graph is nut if and only if its nullity is one.
\end{proposition}

We shall end this section by stating and proving one more proposition that deals with the connectivity of arbitrary circulant graphs.
\begin{proposition}\label{con_property}
    If $\mathrm{Circ}(n, S)$ is any circulant graph, where $S = \{s_0, s_1, \ldots, s_{k-1} \}$ and $1 \le s_0 < s_1 < \cdots <  s_{k-1} \le \frac{n}{2}$, then this graph is connected if and only if
    \[
        \gcd(n, s_0, s_1, \ldots, s_{k-1}) = 1 .
    \]
\end{proposition}
\begin{proof}
    If $\gcd(n, s_0, s_1, \ldots, s_{k-1}) = \beta > 1$, then it is not difficult to notice that by taking any walk that starts at the vertex $0$, it is only possible to reach a vertex $j,\, 0 \le j \le n-1$ satisfying $\beta \mid j$. Hence, the graph cannot be connected. On the other hand, if $\gcd(n, s_0, s_1, \ldots, s_{k-1}) = 1$ and $n \ge 2$, then there certainly exist integers $\alpha_0, \alpha_1, \ldots, \alpha_k \in \mathbb{Z}$ such that
    \begin{equation}\label{aux_connectivity}
        \alpha_0 s_0 + \alpha_1 s_1 + \cdots + \alpha_{k-1} s_{k-1} + \alpha_k n = 1 .
    \end{equation}
    \newpage\noindent
    However, Eq.\ (\ref{aux_connectivity}) practically means that by constructing a walk that starts at the vertex $0$, it is possible to reach the vertex $1$ by taking a $s_0$ jump $\alpha_0$ times, then a $s_1$ jump $\alpha_1$ times, etc. Here, it should be clear that taking a jump a negative number of times would simply mean that the jump should be taken in the opposite direction. By taking this into consideration, it becomes straightforward to deduce that the given graph must be connected, thus completing the proof.
\end{proof}

\section{Explicit nullity formula}\label{sc_the_formula}

From this section onwards, we shall exclusively consider the circulant graphs that are quartic. Now, when dealing with the adjacency matrix of the quartic circulant graph $\mathrm{Circ}(n, \{p, q\}),\, n \ge 5,\, 1 \le p < q < \frac{n}{2}$, it is easy to see that Eq.\ (\ref{general_circulant_p}) quickly transforms to
\begin{equation}\label{quartic_circulant_p}
    P(x) = x^p + x^q + x^{n - q} + x^{n - p}.
\end{equation}
This observation can now be swiftly applied on any quartic circulant graph in order to yield the following conclusion.
\begin{lemma}\label{counting_lemma}
    The nullity of $\mathrm{Circ}(n, \{p, q\}),\, n \ge 5,\, 1 \le p < q < \frac{n}{2}$, is equal to the number of $n$-th roots of unity $\zeta \in \mathbb{C}$ that satisfy the condition
    \[
        \zeta^{p+q} = -1 \quad \lor \quad \zeta^{q-p} = - 1 .
    \]
\end{lemma}
\begin{proof}
    From Eq.\ (\ref{quartic_circulant_p}), we immediately obtain
    \begin{alignat*}{2}
        && P(\zeta) &= 0\\
        \iff \quad && \zeta^p + \zeta^q + \zeta^{n-q} + \zeta^{n-p} &= 0\\
        \iff \quad && \zeta^q \left( \zeta^p + \zeta^q + \zeta^{-q} + \zeta^{-p} \right) &= 0\\
        \iff \quad && \zeta^{2q} + \zeta^{q + p} + \zeta^{q-p} + 1 &= 0\\
        \iff \quad && (\zeta^{q+p} + 1)(\zeta^{q-p} + 1) &= 0 .
    \end{alignat*}
    Thus, the nullity of $\mathrm{Circ}(n, \{p, q\})$ must equal to the number of $n$-th roots of unity $\zeta$ for which $(\zeta^{q+p} + 1)(\zeta^{q-p} + 1) = 0$ is true. The lemma statement follows immediately from here.
\end{proof}

By relying on Lemma \ref{counting_lemma}, we shall now derive an explicit formula for $\eta(\mathrm{Circ}(n, \linebreak \{p, q\}))$ in terms of the parameters $n, p, q$. First of all, let $v_2(x),\, x \in \mathbb{N}$ denote the number of times that the prime two appears as a prime factor of $x$, i.e.\ the unique non-negative integer such that $2^{v_2(x)} \mid x$ and $2^{v_2(x) + 1} \nmid x$. Furthermore, we shall use the auxiliary notation $\Psi(x, y),\, x, y \in \mathbb{N}$ to denote the set of all the $x$-th roots of unity that give $-1$ when raised to the power of $y$. Bearing this in mind, Lemma~\ref{counting_lemma} allows us to express $\eta(\mathrm{Circ}(n, \{p, q\}))$ by using the formula
\[
    \eta(\mathrm{Circ}(n, \{p, q\})) = | \Psi(n, p+q) \cup \Psi(n, q-p) |,
\]
which instantly transforms to
\begin{equation}\label{counting_cool}
    \eta(\mathrm{Circ}(n, \{p, q\})) = | \Psi(n, p+q) | + |\Psi(n, q-p)| - |\Psi(n, p+q) \cap \Psi(n, q-p) |.
\end{equation}
Taking Eq.\ (\ref{counting_cool}) into consideration, we are able to disclose and prove the following theorem that computes $\eta(\mathrm{Circ}(n, \{p, q\}))$ in terms of $n, p, q$, as desired.

\begin{theorem}\label{main_theorem}
    Let $G = \mathrm{Circ}(n, \{p, q\}),\, n \ge 5,\, 1 \le p < q < \frac{n}{2}$ be any quartic circulant graph. If we denote
    \begin{align*}
        \eta^{(1)} &= \gcd(n, p+q), & v_2^{(1)} &= v_2(p+q),\\
        \eta^{(2)} &= \gcd(n, q-p), & v_2^{(2)} &= v_2(q-p),\\
        \eta^{(3)} &= \gcd(n, p+q, q-p), & v_2^{(3)} &= v_2(n),
    \end{align*}
    then $\eta(G)$ can be computed by using the following formula
    \begin{equation}\label{main_nullity_formula}
        \eta(G) = \begin{cases}
            0, & v_2^{(1)} \ge v_2^{(3)} \land v_2^{(2)} \ge v_2^{(3)},\\
            \eta^{(1)}, & v_2^{(1)} < v_2^{(3)} \land v_2^{(2)} \ge v_2^{(3)},\\
            \eta^{(2)}, & v_2^{(1)} \ge v_2^{(3)} \land v_2^{(2)} < v_2^{(3)},\\
            \eta^{(1)} + \eta^{(2)}, & v_2^{(1)} < v_2^{(3)} \land v_2^{(2)} < v_2^{(3)} \land v_2^{(1)} \neq v_2^{(2)},\\
            \eta^{(1)} + \eta^{(2)} - \eta^{(3)}, & v_2^{(1)} < v_2^{(3)} \land v_2^{(2)} < v_2^{(3)} \land v_2^{(1)} = v_2^{(2)}.
        \end{cases}
    \end{equation}
\end{theorem}
\begin{proof}
    First of all, suppose that $n$ is odd. In this case, it is clear that neither $\zeta^{p+q} = -1$ nor $\zeta^{q-p} = -1$ can have any solution among the $n$-th roots of unity $\zeta \in \mathbb{C},\, \zeta^n = 1$ due to the sheer fact that $-1$ is surely not an $n$-th root of unity. For this reason, we immediately obtain $\eta(G) = 0$ by virtue of Lemma \ref{counting_lemma}. This result is in accordance with Eq.\ (\ref{main_nullity_formula}), since $n$ being odd would surely imply $v_2(p+q) \ge v_2(n)$ and $v_2(q-p) \ge v_2(n)$. In the rest of the proof, we shall only consider the scenario when $2 \mid n$.

    Now, if we put $\zeta = e^{\frac{2t \pi}{n} i}$ for a uniquely defined $t \in \mathbb{N}_0,\, 0 \le t < n$, it is straightforward to notice that the equation
    \begin{equation}\label{aux_1}
        \zeta^{p+q} = -1
    \end{equation}
    in $\zeta \in \mathbb{C},\, \zeta^n = 1$, becomes equivalent to the equation
    \begin{equation}\label{aux_2}
        t (p+q) \equiv_n \frac{n}{2}
    \end{equation}
    in $t \in \mathbb{N}_0,\, 0 \le t < n$. Hence, Eqs.\ (\ref{aux_1}) and (\ref{aux_2}) contain the same number of solutions. However, Eq.\ (\ref{aux_2}) is just a linear congruence equation, which means that it has a solution if and only if $\gcd(n, p+q) \mid \frac{n}{2}$, and in case it does contain a solution, the total number of distinct solutions must be $\gcd(n, p+q)$ (see, for example, \cite[pp.\ 170, Theorem 5.14]{Tattersall}). Besides that, it is not difficult to realize that $\gcd(n, p+q) \mid \frac{n}{2}$ is equivalent to $v_2(p + q) < v_2(n)$. By taking everything into consideration, we conclude that
    \begin{equation}\label{aux_3}
        | \Psi(n, p+q) | = \begin{cases}
            \gcd(n, p+q), & v_2(p + q) < v_2(n),\\
            0, & v_2(p + q) \ge v_2(n) .
        \end{cases}
    \end{equation}

    The equation $\zeta^{q-p} = -1$ in $\zeta \in \mathbb{C},\, \zeta^n = 1$, can be analyzed in an entirely analogous manner in order to obtain
    \begin{equation}\label{aux_4}
        | \Psi(n, q-p) | = \begin{cases}
            \gcd(n, q-p), & v_2(q - p) < v_2(n),\\
            0, & v_2(q - p) \ge v_2(n) .
        \end{cases}
    \end{equation}
    
    We shall now focus on computing $| \Psi(n, p+q) \cap \Psi(n, q-p) |$, i.e.\ the total number of solutions of the system of equations
    \begin{align}\label{aux_5}
        \zeta^{p+q} = -1, && \zeta^{q-p} = -1,
    \end{align}
    in $\zeta \in \mathbb{C},\, \zeta^n = 1$. Suppose that $\zeta_0$ is a solution of Eq.\ (\ref{aux_5}). From here it immediately follows that $\zeta_0^{2(p+q)} = 1$ and $\zeta_0^{2(q-p)} = 1$, which further implies
    \begin{alignat*}{2}
        && \zeta_0^{\gcd(2(p+q), 2(q-p))} &= 1\\
        \implies \quad && \zeta_0^{2 \gcd(p+q, q-p)} &= 1\\
        \implies \quad && \zeta_0^{\gcd(p+q, q-p)} &\in \{1, -1\} .
    \end{alignat*}
    Now, the option $\zeta_0^{\gcd(p+q, q-p)} = 1$ is certainly impossible, since this would give $\zeta_0^{p+q} = \zeta_0^{q-p} = 1$. Thus, $\zeta_0^{\gcd(p+q, q-p)} = -1$. Given the fact that
    \begin{align*}
        \zeta_0^{p+q} = -1, && \zeta_0^{q-p} = -1, && \zeta_0^{\gcd(p+q, q-p)} = -1,
    \end{align*}
    it is evident that $\dfrac{p+q}{\gcd(p+q, q-p)}$ and $\dfrac{q-p}{\gcd(p+q, q-p)}$ must be odd. Hence, if $v_2(p+q) \neq v_2(q-p)$, then Eq.\ (\ref{aux_5}) has no solutions. Provided $v_2(p+q) = v_2(q-p)$ is true, it is easy to see that Eq.\ (\ref{aux_5}) becomes equivalent to
    \begin{equation}\label{aux_6}
        \zeta_0^{\gcd(p+q, q-p)} = -1 .
    \end{equation}
    As already discussed, Eq.\ (\ref{aux_6}) has no solutions if $v_2(\gcd(p+q, q-p)) \ge v_2(n)$, and if $v_2(\gcd(p+q, q-p)) < v_2(n)$ does hold, then this equation contains precisely $\gcd(n, \gcd(p+q, q-p))$ solutions. Taking everything into consideration, we get
    \begin{align}\label{aux_7}
    \begin{split}
        | \Psi(n, p+q) &\cap \Psi(n, q-p) | =\\
        &= \begin{cases}
            \gcd(n, p+q, q-p), & v_2(p + q) = v_2(q - p) < v_2(n),\\
            0, & \mbox{otherwise} .
        \end{cases}
    \end{split}
    \end{align}

    Finally, Eq.\ (\ref{main_nullity_formula}) follows by directly using Eq.~(\ref{counting_cool}) together with Eqs.\ (\ref{aux_3}), (\ref{aux_4}) and (\ref{aux_7}).
\end{proof}

By implementing Theorem \ref{main_theorem}, Eq.\ (\ref{main_nullity_formula}) makes it trivial to deduce whether a given quartic circulant graph is singular. The corresponding result is disclosed in the following short corollary whose full proof we choose to omit.
\begin{corollary}\label{singularity_cond}
    A quartic circulant graph $\mathrm{Circ}(n, \{p, q\}),\, n \ge 5, 1 \le p < q < \frac{n}{2}$, is singular if and only if $v_2(p + q) < v_2(n)$ or $v_2(q - p) < v_2(n)$.
\end{corollary}
Bearing in mind Proposition \ref{core_property}, Corollary \ref{singularity_cond} can also be implemented in order to test whether a given quartic circulant graph is core. Besides that, it becomes relatively simple to inspect whether a certain quartic circulant graph is nut. This observation is demonstrated within the next corollary.
\begin{corollary}\label{nullity_one}
    A quartic circulant graph $\mathrm{Circ}(n, \{p, q\}),\, n \ge 5, 1 \le p < q < \frac{n}{2}$, is nut if and only if the following conditions hold:
    \begin{itemize}
        \item $n$ is even;
        \item $p$ and $q$ are of different parities;
        \item $\gcd(n, p+q) = \gcd(n, q-p) = 1$.
    \end{itemize}
\end{corollary}
\begin{proof}
    If $n$ is odd, then Theorem \ref{main_theorem} immediately gives $\eta(\mathrm{Circ}(n, \{p, q\})) = 0$, since $v_2(p + q) \ge v_2(n)$ and $v_2(q-p) \ge v_2(n)$ certainly hold. For this reason, $n$ being odd would imply that the graph is surely not nut.

    Now, suppose that $n$ is even. If $p$ and $q$ were of the same parity, then $p + q$ and $q - p$ would be even, hence $\gcd(n, p + q), \gcd(n, q - p), \gcd(n, p + q, q - p)$ would all be even as well. However, by implementing Theorem \ref{main_theorem}, we now obtain that $\eta(\mathrm{Circ}(n, \{p, q\}))$ would necessarily have to be even, hence the given graph could not be a nut graph, by virtue of Proposition \ref{nut_property}.
    
    Finally, suppose that $p$ and $q$ are of different parities. In this case, it is clear that $v_2(p + q) = v_2(q - p) = 0$, while $v_2(n) \ge 1$. By applying Theorem \ref{main_theorem}, we obtain
    \[
        \eta(\mathrm{Circ}(n, \{p, q\})) = \gcd(n, p + q) + \gcd(n, q - p) - \gcd(n, p + q, q - p) .
    \]
    It is not difficult to notice that
    \begin{align*}
        \eta(\mathrm{Circ}(n, \{p, q\})) \ge \gcd(n, p + q),\\
        \eta(\mathrm{Circ}(n, \{p, q\})) \ge \gcd(n, q - p),
    \end{align*}
    both have to be true. Thus, if $\gcd(n, p + q) > 1$ or $\gcd(n, q - p) > 1$, then $\eta(\mathrm{Circ}(n, \{p, q\}))$ must be greater than one, implying that the graph cannot be nut. Finally, if we have that $\gcd(n, p + q) = 1$ and $\gcd(n, q - p) = 1$ both indeed hold, then Theorem \ref{main_theorem} swiftly yields that the given graph must have the nullity one. Applying Proposition \ref{nut_property} gives that the graph is surely nut, as desired.
\end{proof}

\section{Minimum nullity problem}\label{sc_nullity_min}

In this section, we will disclose the full solution to the quartic circulant graph minimum nullity problem. Let $\mathcal{G}_n,\, n \ge 5$ be the set of all the quartic circulant graphs of order $n$, and let $\mathcal{C}_n \subseteq \mathcal{G}_n$ be the corresponding subset comprising only the connected graphs. The precisely formulated goals of the aforementioned minimum nullity study are to
\begin{itemize}
    \item compute the minimum nullity $\mathcal{N}_\mathrm{min}$,
    \item find all the graphs attaining the minimum nullity,
    \item determine the null spaces of all the aforementioned graphs,
\end{itemize}
for both $\mathcal{G}_n$ and $\mathcal{C}_n$, for each $n \ge 5$. To start with, we will provide the solution to the minimum nullity problem on $\mathcal{G}_n$.

\begin{theorem}\label{min_nullity_g_th}
    For a given positive integer $n \in \mathbb{N},\, n \ge 5$, the minimum nullity that a quartic circulant graph of order $n$ can achieve is given in the second column of Table \ref{min_nullity_g_table}. Also, the graphs $\mathrm{Circ}(n, \{p, q\}), \, 1 \le p < q < \frac{n}{2}$, attaining the minimum nullity are characterized by the conditions displayed in the third column, while their corresponding null spaces are disclosed in the fourth column.
\begin{table}[ht]
\begin{center}
{\scriptsize
\begin{tabular}{llll}
\toprule order & $\mathcal{N}_\mathrm{min}$ & graphs attaining $\mathcal{N}_\mathrm{min}$ & corresponding null space \\
\midrule
$n \ge 5, \, 2 \nmid n$ & $0$ & $p$ and $q$ are arbitrary & trivial\\
\midrule
$n = 6$ & $3$ & \begin{tabular}{@{}l@{}}$p = 1$ and $q = 2$ \\ (this is the only graph)\end{tabular} & $\mathrm{span}\left( \begin{bmatrix} 1\\ 0\\ 0\\ -1\\ 0\\ 0 \end{bmatrix}, \begin{bmatrix} 0\\ 1\\ 0\\ 0\\ -1\\ 0 \end{bmatrix}, \begin{bmatrix} 0\\ 0\\ 1\\ 0\\ 0\\ -1 \end{bmatrix} \right)$\\
\midrule
\begin{tabular}{@{}l@{}}$n = 2^\alpha$ for some $\alpha \ge 3$, or \\ $n = 3 \cdot 2^\alpha$ for some $\alpha \ge 2$\end{tabular} & $1$ & \begin{tabular}{@{}l@{}}$p \not\equiv_2 q$ and \\ $\gcd(n, p+q) = \gcd(n, q-p) = 1$\end{tabular} & $\mathrm{span}\left( \begin{bmatrix} 1\\ -1\\ 1\\ -1\\ \vdots \\ 1\\ -1 \end{bmatrix} \right)$\\
\midrule
$n = \beta \, 2^\alpha$ for some & \multirow{2}{*}{$0$} & $p = \gamma \, 2^{\alpha-1}$ and $q = \delta \, 2^{\alpha-1}$ & \multirow{2}{*}{trivial}\\
$\alpha \ge 1, \, \beta \ge 5,\, 2 \nmid \beta$ & & for some $1 \le \gamma < \delta < \beta,\, \gamma \equiv_2 \delta$ & \\
\bottomrule
\end{tabular}
}
\caption{The complete solution to the minimum nullity problem on $\mathcal{G}_n$.}      
\label{min_nullity_g_table}
\end{center}
\end{table}
\end{theorem}
\begin{proof}
    If $n$ is odd, then Theorem \ref{main_theorem} immediately claims that each quartic circulant graph of order $n$ must necessarily have the nullity zero. Thus, in this case, the minimum nullity is clearly zero and every single graph attains it. Of course, all the null spaces of these graphs must be trivial. These observations correspond to the first row in Table \ref{min_nullity_g_table}. Throughout the rest of the proof, we shall take $n$ to be even.

    It is obvious that there exists a single quartic circulant graph of order six --- $\mathrm{Circ}(6, \{1, 2\})$. Its adjacency matrix is
    \[
        \begin{bmatrix}
            0 & 1 & 1 & 0 & 1 & 1\\
            1 & 0 & 1 & 1 & 0 & 1\\
            1 & 1 & 0 & 1 & 1 & 0\\
            0 & 1 & 1 & 0 & 1 & 1\\
            1 & 0 & 1 & 1 & 0 & 1\\
            1 & 1 & 0 & 1 & 1 & 0
        \end{bmatrix}
    \]
    and it is straightforward to check that its null space is given by
    \[
        \mathrm{span}\left( \begin{bmatrix} 1\\ 0\\ 0\\ -1\\ 0\\ 0 \end{bmatrix}, \begin{bmatrix} 0\\ 1\\ 0\\ 0\\ -1\\ 0 \end{bmatrix}, \begin{bmatrix} 0\\ 0\\ 1\\ 0\\ 0\\ -1 \end{bmatrix} \right) .
    \]
    This means that the given graph has the nullity three. From here, it follows that the minimum nullity for the case $n = 6$ is equal to three and it is attained by the single quartic circulant graph of order six --- the graph determined by $p = 1$ and $q = 2$. These conclusions correspond to the second row in Table \ref{min_nullity_g_table}.

    \newpage
    Suppose that $n \ge 8$. Now, let $n = \beta \, 2^\alpha$, where $\alpha = v_2(n)$ and $2 \nmid b$. By virtue of Corollary \ref{singularity_cond}, we see that a quartic circulant graph $\mathrm{Circ}(n, \{p, q\}),\, 1 \le p < q < \frac{n}{2}$, of order $n$ is of nullity zero if and only if $v_2(p + q) \ge \alpha$ and $v_2(q - p) \ge \alpha$. If we suppose that a given graph $\mathrm{Circ}(n, \{p, q\})$ does have the nullity zero, we are immediately able to obtain that
    $2^\alpha \mid p + q$ and $2^\alpha \mid q - p$, which then gives
    \[
        2^\alpha \mid (p + q) - (q - p) \quad \implies \quad 2^\alpha \mid 2p \quad \implies \quad 2^{\alpha-1} \mid p,
    \]
    alongside
    \[
        2^\alpha \mid (p + q) + (q - p) \quad \implies \quad 2^\alpha \mid 2q \quad \implies \quad 2^{\alpha-1} \mid q .
    \]
    Let $p = \gamma \, 2^{\alpha - 1}$ and $q = \delta \, 2^{\alpha-1}$. From $2^\alpha \mid p + q$, it is easy to notice that $\gamma$ and $\delta$ have to be of the same parity. The converse is also straightforward to show by implementing Corollary \ref{singularity_cond} --- if $p = \gamma \, 2^{\alpha - 1}$ and $q = \delta \, 2^{\alpha - 1}$ for some $\gamma \equiv_2 \delta$, then $\mathrm{Circ}(n, \{p, q\})$ does have the nullity zero.
    
    Now, the question remains whether the values $\gamma$ and $\delta$ can be chosen in order to yield the desired $p$ and $q$ parameters accordingly. In order to give an answer to this question, it is essential to realize that the condition $1 \le p < q < \frac{n}{2}$ quickly converts to $1 \le \gamma < \delta < \beta$. If $\beta \ge 5$, then it is easy to see that the values $\gamma$ and $\delta$ can indeed be picked in at least one valid way. This means that in this scenario, the minimum nullity that a quartic circulant graph of order $n$ can possess is equal to zero. Furthermore, all the graphs that attain the said nullity can be spawned in the aforementioned manner by relying on any two values $\gamma$ and $\delta$ such that $1 \le \gamma < \delta < \beta$ and $\gamma \equiv_2 \delta$. Clearly, all of their null spaces must be trivial. This case corresponds to the fourth row in Table \ref{min_nullity_g_table}.

    Finally, suppose that $\beta = 1$ or $\beta = 3$. The condition $1 \le \gamma < \delta < \beta$ now becomes impossible to satisfy, which means that in this case there does not exist a quartic circulant graph of order $n$ whose nullity is zero. However, we do know that for each even $n \ge 8$, there surely exists a quartic circulant graph of order $n$ with nullity one, as shown by Damnjanovi\'c \cite[Theorem 5]{Damnjanovic3}. For this reason, in this last case, the desired minimum nullity must be equal to one. Logically, we now have that the graphs attaining the minimum nullity are precisely those that are nut, due to Proposition \ref{nut_property}. Hence, we obtain their full characterization by directly implementing Corollary \ref{nullity_one}. Besides that, we immediately get that all of their null spaces must be given by
    \[
        \mathrm{span}\left( \begin{bmatrix} 1 & -1 & 1 & -1 & \cdots & 1 & -1 \end{bmatrix}^\mathrm{T} \right) ,
    \]
    as shown by Damnjanovi\'c and Stevanovi\'c \cite[Remark 5]{Damnjanovic1}. This final case corresponds to the third row in Table \ref{min_nullity_g_table}.
\end{proof}

We now turn our focus to the minimum nullity problem on $\mathcal{C}_n$, i.e.\ solving the same problem as we have in Theorem \ref{min_nullity_g_th}, but with the additional constraint that all the graphs considered must be connected. We are able to quickly resolve the said problem by using a technique similar to the one previously used in the aforementioned theorem. The corresponding results are disclosed in the following theorem.
\begin{theorem}\label{min_nullity_c_th}
    For a given positive integer $n \in \mathbb{N},\, n \ge 5$, the minimum nullity that a connected quartic circulant graph of order $n$ can achieve is given in the second column of Table \ref{min_nullity_c_table}. Also, the graphs $\mathrm{Circ}(n, \{p, q\}), \, 1 \le p < q < \frac{n}{2}$, attaining the minimum nullity are characterized by the conditions displayed in the third column, while their corresponding null spaces are disclosed in the fourth column.
\begin{table}[ht]
\begin{center}
{\scriptsize
\begin{tabular}{llll}
\toprule order & $\mathcal{N}_\mathrm{min}$ & graphs attaining $\mathcal{N}_\mathrm{min}$ & corresponding null space \\
\midrule
$n \ge 5, \, 2 \nmid n$ & $0$ & $\gcd(p, q, n) = 1$ & trivial\\
\midrule
$n = 6$ & $3$ & \begin{tabular}{@{}l@{}}$p = 1$ and $q = 2$ \\ (this is the only graph)\end{tabular} & $\mathrm{span}\left( \begin{bmatrix} 1\\ 0\\ 0\\ -1\\ 0\\ 0 \end{bmatrix}, \begin{bmatrix} 0\\ 1\\ 0\\ 0\\ -1\\ 0 \end{bmatrix}, \begin{bmatrix} 0\\ 0\\ 1\\ 0\\ 0\\ -1 \end{bmatrix} \right)$\\
\midrule
$n \ge 8, \, 4 \mid n$ & $1$ & \begin{tabular}{@{}l@{}}$p \not\equiv_2 q$ and \\ $\gcd(n, p+q) = \gcd(n, q-p) = 1$\end{tabular} & $\mathrm{span}\left( \begin{bmatrix} 1\\ -1\\ 1\\ -1\\ \vdots \\ 1\\ -1 \end{bmatrix} \right)$\\
\midrule
$n \ge 10, \, n \equiv_4 2$ & $0$ & $2 \nmid p, q$ and $\gcd(p, q, n) = 1$ & trivial\\
\bottomrule
\end{tabular}
}
\caption{The complete solution to the minimum nullity problem on $\mathcal{C}_n$.}      
\label{min_nullity_c_table}
\end{center}
\end{table}
\end{theorem}
\begin{proof}
    First of all, by implementing Proposition \ref{con_property}, we see that a quartic circulant graph $\mathrm{Circ}(n, \{p, q\})$ is connected if and only if $\gcd(n, p, q) = 1$. Now, if $n$ is odd, then Theorem \ref{main_theorem} implies that each connected quartic circulant graph of order $n$ has the nullity zero. The results displayed in the first row of Table \ref{min_nullity_c_table} follow immediately from here. Also, the edge case when $n = 6$ can be proved in an absolutely identical manner as done so in Theorem \ref{min_nullity_g_th}. Of course, this scenario corresponds to the second row of Table \ref{min_nullity_c_table}. Thus, we can suppose that $n \ge 8$ and $2 \mid n$ in the remainder of the proof.

    If $4 \mid n$, then it is not difficult to demonstrate that there does not exist a connected quartic circulant graph of order $n$ whose nullity is zero. Suppose that such a graph exists in the form of $\mathrm{Circ}(n, \{p, q\}), \, 1 \le p < q < \frac{n}{2}$. Since $v_2(n) \ge 2$, Corollary \ref{singularity_cond} tells us that $v_2(p + q) \ge 2$ and $v_2(q - p) \ge 2$ must both hold, which further implies
    \[
        4 \mid (p + q) - (q - p) \quad \implies \quad 4 \mid 2p \quad \implies \quad 2 \mid p,
    \]
    together with
    \[
        4 \mid (p + q) + (q - p) \quad \implies \quad 4 \mid 2q \quad \implies \quad 2 \mid q .
    \]
    However, this means that $n, p, q$ are surely all even, which yields $\gcd(n, p, q) \ge 2$, contradicting the fact that $\mathrm{Circ}(n, \{p, q\})$ is connected. Thus, each connected quartic circulant graph of order $n$ must be singular. Now, we know that there certainly exists a quartic circulant nut graph of order $n$, as shown by Damnjanovi\'c \cite[Theorem 5]{Damnjanovic3}. Besides that, it is simple to see that each nut graph must be connected, for if it had multiple components, then precisely one of them would have to be singular, which would quickly yield a non-full non-zero null space vector. Bearing this in mind, we get that the minimum nullity attainable by a graph from $\mathcal{C}_n$ is equal to one and the graphs attaining it are precisely those that are nut. Their null spaces are necessarily obtained via
    \[
        \mathrm{span}\left( \begin{bmatrix} 1 & -1 & 1 & -1 & \cdots & 1 & -1 \end{bmatrix}^\mathrm{T} \right) ,
    \]
    as discussed earlier. The noted observations correspond to the third row of Table~\ref{min_nullity_c_table}.

    Finally, suppose that $n \equiv_4 2$ and $n \ge 10$. In this case, it is not difficult to see that a graph $\mathrm{Circ}(n, \{p, q\})$ is of nullity zero if and only if $v_2(p + q) \ge 1$ and $v_2(q - p) \ge 1$, by virtue of Corollary \ref{singularity_cond}. In other words, this graph has the nullity zero if and only if $p$ and $q$ are of the same parity. Of course, it is impossible for $p$ and $q$ to be even, since this would mean that $\gcd(n, p, q) \ge 2$, contradicting the fact that $\mathrm{Circ}(n, \{p, q\}) \in \mathcal{C}_n$. Thus, a quartic circulant graph $\mathrm{Circ}(n, \{p, q\})$ is connected and has the nullity zero if and only if $2 \nmid p, q$ and $\gcd(n, p, q) = 1$. It is obvious that such graphs exist --- for example, $\mathrm{Circ}(n, \{1, 3\})$ satisfies both of these criteria. From here, it immediately follows that the minimum nullity attainable on $\mathcal{C}_n$ is equal to zero and that the graphs attaining it all possess a trivial null space, as disclosed in Table \ref{min_nullity_c_table}.
\end{proof}

\section{Maximum nullity problem}\label{sc_nullity_max}

We shall end the paper by giving a full solution to the quartic circulant graph maximum nullity problem. In a similar manner as in Section \ref{sc_nullity_min}, the goals of the aforementioned study will be to
\newpage
\begin{itemize}
    \item compute the maximum nullity $\mathcal{N}_\mathrm{max}$,
    \item find all the graphs attaining the maximum nullity,
    \item determine the null spaces of all the aforementioned graphs,
\end{itemize}
for both $\mathcal{G}_n$ and $\mathcal{C}_n$, for each $n \ge 5$. In order to concisely carry out the desired study, we will be in need of the following three auxiliary lemmas.

\begin{lemma}\label{max_aux_lemma_1}
    A quartic circulant graph $G = \mathrm{Circ}(n, \{p, q\}),\, 1 \le p < q < \frac{n}{2}$ where $n \ge 6, \, 2 \mid n, \, 8 \nmid n$ and $p + q = \frac{n}{2}$, has the nullity $\eta(G) = \frac{n}{2}$ and its null space can be described via the expression
    \begin{equation}\label{max_aux_lemma_formula}
        \mathcal{N}(G) = \left\{ \begin{bmatrix} v_0 & v_1 & \cdots & v_{n-1} \end{bmatrix}^\mathrm{T} \colon v_{j + \frac{n}{2}} = -v_j \ \left( \forall j = \overline{0, \frac{n}{2}-1} \right) \right\} . 
    \end{equation}
\end{lemma}
\begin{proof}
    The elements of $\mathcal{N}(G)$ are actually the solutions to the equation
    \begin{equation}\label{max_aux_1}
        A_G \, u = 0
    \end{equation}
    in $u = \begin{bmatrix} u_0 & u_1 & \cdots & u_{n-1} \end{bmatrix}^\mathrm{T}$. However, Eq.\ (\ref{max_aux_1}) is directly equivalent to
    \[
        u_{j+p} + u_{j+q} + u_{j+n-q} + u_{j+n-p} = 0 \qquad (\forall j = \overline{0, n-1}),
    \]
    where the computation of indices is done modulo $n$. Now, if we put $w_j = u_j + u_{j + \frac{n}{2}}$ for each $j = \overline{0, \frac{n}{2}-1}$, we quickly obtain
    \[
        u_{j+p} + u_{j+q} + u_{j+n-q} + u_{j+n-p} =  u_{j+p} + u_{j+q} + u_{j+\frac{n}{2} + p} + u_{j+\frac{n}{2} + q} = w_{j + p} + w_{j + q} ,
    \]
    which means that Eq.\ (\ref{max_aux_1}) gets down to solving the system of equations
    \begin{equation}\label{max_aux_2}
        w_j + w_{j + q - p} = 0 \qquad \left( \forall j = \overline{0, \frac{n}{2}-1} \right)
    \end{equation}
    in $w = \begin{bmatrix} w_0 & w_1 & \cdots & w_{\frac{n}{2}-1} \end{bmatrix}^\mathrm{T}$, with the indices being computed modulo $\frac{n}{2}$. 
    
    Now, it is straightforward to notice that Eq.\ (\ref{max_aux_2}) would certainly imply
    \begin{equation}\label{max_aux_3}
        w_{j + \alpha(q - p)} = (-1)^\alpha \, w_j \qquad \left( \forall j = \overline{0, \frac{n}{2}-1},\, \forall \alpha \in \mathbb{Z} \right).
    \end{equation}
    \newpage\noindent
    Bearing in mind that $p + q = \frac{n}{2}$ and $q - p$ are surely of the same parity, together with $8 \nmid n$, we conclude that either $v_2(q-p) = v_2 \left( \frac{n}{2} \right) = 0$ or $v_2(q-p) \ge 1 = v_2\left( \frac{n}{2} \right)$. This means that the value $\dfrac{\frac{n}{2}}{\gcd\left(\frac{n}{2}, q-p\right)}$ must be odd, hence if we put $\alpha = \dfrac{\frac{n}{2}}{\gcd\left(\frac{n}{2}, q-p\right)}$, it follows that
    \[
        \alpha(q - p) = \dfrac{\frac{n}{2}}{\gcd\left(\frac{n}{2}, q-p\right)} \cdot (q - p) = \frac{n}{2} \cdot \dfrac{q - p}{\gcd\left(\frac{n}{2}, q-p\right)} \quad \implies \quad \frac{n}{2} \mid \alpha(q-p),
    \]
    together with
    \[
        (-1)^\alpha w_j = -w_j .
    \]
    Thus, Eq.\ (\ref{max_aux_3}) swiftly transforms to $w_j = -w_j$. For this reason, we get that the only solution to Eq.\ (\ref{max_aux_2}) is achieved when all the $w_j$ values are zero. Now, if we go back to solving Eq.\ (\ref{max_aux_1}), we deduce that its set of solutions coincides precisely with the vector subspace given in Eq.\ (\ref{max_aux_lemma_formula}). It is obvious that the obtained null space is $\frac{n}{2}$-dimensional, which completes the proof.
\end{proof}

\begin{lemma}\label{max_aux_lemma_2}
    A connected quartic circulant graph $G = \mathrm{Circ}(n, \{p, q\}),\, 1 \le p < q < \frac{n}{2}$ where $8 \mid n$ and $p + q = \frac{n}{2}$, has the nullity $\eta(G) = \frac{n}{2} + 2$ and its null space can be described via the expression
    \begin{align}\label{max_aux_lemma_formula_2}
        \begin{split}
        \mathcal{N}(G) = \mathrm{span} \bigg( & e_0 - e_\frac{n}{2}, e_1 - e_{\frac{n}{2} + 1}, e_2 - e_{\frac{n}{2} + 2}, \ldots,  e_{\frac{n}{2} - 1} - e_{n-1}, \\
        &\sum_{j = 0}^{\frac{n}{4}-1} (-1)^j \, e_{\frac{n}{2} + 2j}, \sum_{j = 0}^{\frac{n}{4}-1} (-1)^j \, e_{\frac{n}{2} + 2j+1} \bigg) ,
        \end{split}
    \end{align}
    where $e_j \in \mathbb{R}^n$ represents the vector whose $j$-th element equals one, while all the others are zero.
\end{lemma}
\begin{proof}
    The elements of $\mathcal{N}(G)$ are the solutions to the equation Eq.\ (\ref{max_aux_1}), which can quickly be transformed into Eq.\ (\ref{max_aux_2}) in an absolutely identical manner as it was done in Lemma \ref{max_aux_lemma_1}. Now, due to the fact that $p + q$ is even, it is clear that $p$ and $q$ must be of the same parity. However, it is impossible for $p$ and $q$ to be even, since this would imply $\gcd(n, p, q) \ge 2$, contradicting the fact that $G$ is connected, in accordance with Proposition \ref{con_property}. This means that $p$ and $q$ must both be odd. Due to the fact that $4 \mid p + q$, it becomes straightforward to see that $q - p \equiv_4 2$.

    Since we know that
    \begin{align*}
        \gcd\left(q-p, \frac{n}{2}\right) &\mid (p + q) + (q - p)\quad \implies \quad \gcd\left(q-p, \frac{n}{2}\right) \mid 2q,\\
        \gcd\left(q-p, \frac{n}{2}\right) &\mid (p + q) - (q - p)\quad \implies \quad \gcd\left(q-p, \frac{n}{2}\right) \mid 2p,\\
        \gcd\left(q-p, \frac{n}{2}\right) &\mid 2n,
    \end{align*}
    we immediately conclude that $\gcd\left(q-p, \frac{n}{2}\right) \mid \gcd(2n, 2p, 2q)$. However, the connectedness of $G$ implies $\gcd(n, p, q) = 1$, while $q-p$ and $\frac{n}{2}$ are both clearly even, which means that $\gcd\left(q-p, \frac{n}{2}\right) = 2$ must be true. For this reason, there must exist some integers $\alpha, \beta \in \mathbb{Z}$ such that
    \begin{equation}\label{max_aux_4}
        \alpha (q - p) + \beta \, \frac{n}{2} = 2 .
    \end{equation}
    Of course, $\alpha$ needs to be odd, for if it were even, then the left-hand side of Eq.~(\ref{max_aux_4}) would be divisible by four, which is not possible. Now, Eq.\ (\ref{max_aux_2}) would certainly imply Eq.\ (\ref{max_aux_3}), and by plugging in the aforementioned value of $\alpha$, we reach
    \begin{equation}\label{max_aux_5}
        w_{j + 2} = -w_j \qquad \left( \forall j = \overline{0, \frac{n}{2}-1} \right) .
    \end{equation}
    Due to $q - p \equiv_4 2$, the converse is straightforward to check, hence Eq.\ (\ref{max_aux_2}) is necessarily equivalent to Eq.\ (\ref{max_aux_5}).

    Since $4 \mid \frac{n}{2}$, it is trivial to notice that the set of solutions to the system of equations Eq.\ (\ref{max_aux_5}) in $w = \begin{bmatrix} w_0 & w_1 & \cdots & w_{\frac{n}{2}-1} \end{bmatrix}^\mathrm{T}$ is given by
    \begin{alignat*}{2}
        w_{2j} &= (-1)^j \, C_1 \qquad && \left( \forall j = \overline{0, \frac{n}{4} - 1} \right),\\
        w_{2j + 1} &= (-1)^j \, C_2 \qquad && \left( \forall j = \overline{0, \frac{n}{4} - 1} \right),
    \end{alignat*}
    where $C_1, C_2 \in \mathbb{R}$ are two arbitrary constants. By going back to the original system Eq.\ (\ref{max_aux_1}), it is not difficult to see that solution set can be thought of as
    \begin{alignat*}{2}
        u_j &= u_j \qquad & \left(\forall j = \overline{0, \frac{n}{2}-1}\right),\\
        u_{\frac{n}{2} + 2j} &= (-1)^j \, C_1 - u_{2j} \qquad & \left(\forall j = \overline{0, \frac{n}{4}-1}\right),\\
        u_{\frac{n}{2} + 2j + 1} &= (-1)^j \, C_2 - u_{2j + 1} \qquad & \left(\forall j = \overline{0, \frac{n}{4}-1}\right),
    \end{alignat*}
    where $u_0, u_1, \ldots, u_{\frac{n}{2}-1}, C_1, C_2$ are regarded as arbitrary constants. From here, it immediately follows that the given solution set coincides with the vector subspace given in Eq.\ (\ref{max_aux_lemma_formula_2}). Also, it is simple to realize that the obtained null space must be of dimension $\frac{n}{2} + 2$, given the fact that the $\frac{n}{2} + 2$ vectors given in Eq.\ (\ref{max_aux_lemma_formula_2}) form a linearly independent set.
\end{proof}

\begin{lemma}\label{max_aux_lemma_3}
    The quartic circulant graph $G = \mathrm{Circ}\left(n, \left\{\frac{n}{8}, \frac{3n}{8}\right\} \right)$, where $8 \mid n$, has the nullity $\eta(G) = \frac{3n}{4}$ and its null space can be described via the expression
    \begin{equation}\label{max_aux_lemma_formula_3}
        \mathcal{N}(G) = \left\{ \begin{bmatrix} v_0\\ v_1\\ \vdots\\ v_{n-1} \end{bmatrix} \colon v_{j + \frac{3n}{4}} = -v_j - v_{j+\frac{n}{4}} - v_{j+\frac{n}{2}} \ \left( \forall j = \overline{0, \frac{n}{4}-1} \right) \right\} .
    \end{equation}
\end{lemma}
\begin{proof}
    Once more, the elements of $\mathcal{N}(G)$ represent the solutions to the system of equations Eq.~(\ref{max_aux_1}). In this scenario, the given system transforms to
    \[
        u_{j + \frac{n}{8}} + u_{j+\frac{3n}{8}} + u_{j+\frac{5n}{8}} + u_{j+\frac{7n}{8}} = 0 \qquad (\forall j = \overline{0, n-1}),
    \]
    which is equivalent to
    \[
        u_j + u_{j+\frac{n}{4}} + u_{j+\frac{n}{2}} + u_{j+\frac{3n}{4}} = 0 \qquad (\forall j = \overline{0, n-1}) .
    \]
    From here, it becomes evident that the solution set must be equal to the vector subspace given in Eq.\ (\ref{max_aux_lemma_formula_3}). Also, it is obvious that $\mathcal{N}(G)$ is of dimension $\frac{3n}{4}$, which completes the proof.
\end{proof}

We now turn our attention to the maximum nullity problem on the set $\mathcal{G}_n$ and disclose the complete solution to the said problem in the form of the following theorem.

\begin{theorem}\label{max_nullity_g_th}
    For a given positive integer $n \in \mathbb{N},\, n \ge 5$, the maximum nullity that a quartic circulant graph of order $n$ can achieve is given in the second column of Table \ref{max_nullity_g_table}. Also, the graphs $\mathrm{Circ}(n, \{p, q\}), \, 1 \le p < q < \frac{n}{2}$, attaining the maximum nullity are characterized by the conditions displayed in the third column, while their corresponding null spaces are disclosed in the fourth column.

\begin{table}[ht]
\begin{center}
{\scriptsize
\begin{tabular}{llll}
\toprule order & $\mathcal{N}_\mathrm{max}$ & graphs attaining $\mathcal{N}_\mathrm{max}$ & corresponding null space \\
\midrule
$n \ge 5, \, 2 \nmid n$ & $0$ & $p$ and $q$ are arbitrary & trivial\\
\midrule
$n \ge 6, \, 2 \mid n,\, 8 \nmid n$ & $\frac{1}{2}n$ & $p + q = \frac{n}{2}$ & see Lemma \ref{max_aux_lemma_1}\\
\midrule
$n \ge 8, \, 8 \mid n$ & $\frac{3}{4}n$ & $p = \frac{1}{8}n$ and $q = \frac{3}{8}n$ & see Lemma \ref{max_aux_lemma_3}\\
\bottomrule
\end{tabular}
}
\caption{The complete solution to the maximum nullity problem on $\mathcal{G}_n$.}      
\label{max_nullity_g_table}
\end{center}
\end{table}
\end{theorem}
\begin{proof}
    Similarly as in the proofs of Theorems \ref{min_nullity_g_th} and \ref{min_nullity_c_th}, we have that if $n$ is odd, then each quartic circulant graph of order $n$ surely has the nullity zero. This observation quickly proves the validity of the first row in Table \ref{max_nullity_g_table}.

    \newpage
    Suppose that $n$ is even and let $G = \mathrm{Circ}(n, \{p, q\})$ be a given graph. We shall now prove that $\eta(G) \le \frac{n}{2}$ if $8\nmid n$ and $\eta(G) \le \frac{3n}{4}$ if $8 \mid n$. Afterwards, we will quickly find all the graphs attaining the maximum nullity and determine their corresponding null spaces. In order to make the remainder of the proof more concise, we will split the analysis of $G$ into four cases.

    \bigskip\noindent
    \emph{Case $v_2(p + q) \ge v_2(n)$ and $v_2(q - p) \ge v_2(n)$}.\quad
    In this case, Theorem \ref{main_theorem} claims that $\eta(G) = 0$, hence $\eta(G) < \frac{n}{2} < \frac{3n}{4}$.

    \bigskip\noindent
    \emph{Case $v_2(p + q) < v_2(n)$ and $v_2(q - p) \ge v_2(n)$}.\quad
    Here, Theorem \ref{main_theorem} directly gives us $\eta(G) = \gcd(n, p + q)$. Since $v_2(p + q) < v_2(n)$, we further get
    \begin{equation}\label{aux_9}
        \eta(G) = \gcd\left(\frac{n}{2}, p + q \right) \le \frac{n}{2} .
    \end{equation}
    Moreover, since $1 \le p < q < \frac{n}{2}$, it is easy to check that the only way for the equality in Eq.\ (\ref{aux_9}) to hold is if $p + q = \frac{n}{2}$. For this reason, we conclude that in the given case, either $\eta(G) < \frac{n}{2}$, or $\eta(G) = \frac{n}{2}$ alongside $p + q = \frac{n}{2}$.

    \bigskip\noindent
    \emph{Case $v_2(p + q) \ge v_2(n)$ and $v_2(q - p) < v_2(n)$}.\quad
    In this case, Theorem \ref{main_theorem} claims that $\eta(G) = \gcd(n, q - p)$. Similarly as in the previous case, we obtain
    \begin{equation}\label{aux_8}
        \eta(G) = \gcd\left(\frac{n}{2}, q - p \right) \le \frac{n}{2} .
    \end{equation}
    However, due to the fact that $1 \le q - p \le \frac{n}{2} - 1$, it is obvious that the equality can never hold in Eq.\ (\ref{aux_8}). Hence, in this case, we get that the nullity is strictly smaller than $\frac{n}{2}$.

    \bigskip\noindent
    \emph{Case $v_2(p + q) < v_2(n)$ and $v_2(q - p) < v_2(n)$}.\quad
    First of all, suppose that $8 \nmid n$. In this scenario, we have that $v_2(n) \le 2$, which means that $0 \le v_2(p + q), v_2(q - p) \le 1$. However, since $p + q$ and $q - p$ are of the same parity, we obtain that either $v_2(p + q) = v_2(q - p) = 1$ or $v_2(p + q) = v_2(q - p) = 0$. Now, by implementing Theorem \ref{main_theorem}, we get that
    \[
        \eta(G) = \gcd(n, p+q) + \gcd(n, q-p) - \gcd(n, p+q, q-p) .
    \]
    As discussed in the previous cases, we have that $\gcd(n, p + q) = \gcd\left(\frac{n}{2}, p+q\right)$ and $\gcd(n, q-p) = \gcd\left(\frac{n}{2}, q-p\right)$. Also, $\gcd\left(\frac{n}{2}, p + q\right) = \frac{n}{2}$ if and only if $p + q = \frac{n}{2}$, while $\gcd\left(\frac{n}{2}, q-p\right) \neq \frac{n}{2}$.
    Besides that, $\gcd\left(\frac{n}{2}, p+q\right)$ and $\gcd\left(\frac{n}{2}, q-p\right)$ are both surely divisors of $\frac{n}{2}$. Taking all of this into consideration, we get that if $p + q \neq \frac{n}{2}$, then
    \[
        \eta(G) \le \frac{1}{2} \cdot \frac{n}{2} + \frac{1}{2} \cdot \frac{n}{2} - \gcd(n, p+q, q-p) = \frac{n}{2} - \gcd(n, p+q, q-p) < \frac{n}{2},
    \]
    while for $p + q = \frac{n}{2}$, we have
    \begin{align*}
        \eta(G) &= \frac{n}{2} + \gcd\left(\frac{n}{2}, q-p\right) - \gcd\left(n, \frac{n}{2}, q-p \right)\\
        &= \frac{n}{2} + \gcd\left(\frac{n}{2}, q-p\right) -  \gcd\left(\frac{n}{2}, q-p\right) = \frac{n}{2} .
    \end{align*}
    Thus, if $8 \nmid n$, then either $\eta(G) < \frac{n}{2}$, or $\eta(G) = \frac{n}{2}$ together with $p + q = \frac{n}{2}$.
    
    Now, suppose that $8 \mid n$ instead. Here, we analogously get that 
    \begin{align*}
        \gcd(n, p + q) &= \gcd\left(\frac{n}{2}, p+q\right) \le \frac{n}{2},\\
        \gcd(n, q - p) &= \gcd\left(\frac{n}{2}, q-p\right) \le \frac{n}{4},
    \end{align*}
    which immediately gives $\eta(G) \le \frac{3n}{4}$. If the equality $\eta(G) = \frac{3n}{4}$ does hold, then it is not diffuclt to see that $p + q = \frac{n}{2}$ and $q - p = \frac{n}{4}$ must be true, which further gives $p = \frac{3n}{8}$ and $q = \frac{n}{8}$.
    
    \bigskip
    Taking everything into consideration, we get that if $8 \nmid n$, then $\eta(G) \le \frac{n}{2}$, with $\eta(G) = \frac{n}{2}$ surely implying $p + q = \frac{n}{2}$. By virtue of Lemma \ref{max_aux_lemma_1}, we see that whenever $p + q = \frac{n}{2}$, the graph nullity must surely be equal to $\frac{n}{2}$, with the null space necessarily being equal to the one given in Eq.\ (\ref{max_aux_lemma_formula}). These observations correspond to the second row in Table \ref{max_nullity_g_table}.

    Finally, if $8 \mid n$, then $\eta(G) \le \frac{3n}{4}$, with $\eta(G) = \frac{3n}{4}$ necessarily implying $p = \frac{n}{8}$ and $q = \frac{3n}{8}$. However, if we do put $p = \frac{n}{8}$ and $q = \frac{3n}{8}$, then the nullity is certainly equal to $\frac{3n}{4}$, with the null space being described via Eq.\ (\ref{max_aux_lemma_formula_3}), by virtue of Lemma~\ref{max_aux_lemma_3}. The noted conclusion is disclosed in the third row in Table \ref{max_nullity_g_table}.
\end{proof}

Finally, we provide the solution to the maximum nullity problem while dealing with connected quartic circulant graphs only. The corresponding result is given in the next theorem.
\begin{theorem}
    For a given positive integer $n \in \mathbb{N},\, n \ge 5$, the maximum nullity that a connected quartic circulant graph of order $n$ can achieve is given in the second column of Table \ref{max_nullity_c_table}. Also, the graphs $\mathrm{Circ}(n, \{p, q\}), \, 1 \le p < q < \frac{n}{2}$, attaining the maximum nullity are characterized by the conditions displayed in the third column, while their corresponding null spaces are disclosed in the fourth column.

\begin{table}[ht]
\begin{center}
{\scriptsize
\begin{tabular}{llll}
\toprule order & $\mathcal{N}_\mathrm{max}$ & graphs attaining $\mathcal{N}_\mathrm{max}$ & corresponding null space \\
\midrule
$n \ge 5, \, 2 \nmid n$ & $0$ & $\gcd(p, q, n) = 1$ & trivial\\
\midrule
$n \ge 6, \, 2 \mid n,\, 8 \nmid n$ & $\frac{1}{2}n$ & $p + q = \frac{n}{2}$ and $\gcd(p, q, n) = 1$ & see Lemma \ref{max_aux_lemma_1}\\
\midrule
$n \ge 8, \, 8 \mid n$ & $\frac{1}{2}n + 2$ & $p + q = \frac{n}{2}$ and $\gcd(p, q, n) = 1$ & see Lemma \ref{max_aux_lemma_2}\\
\bottomrule
\end{tabular}
}
\caption{The complete solution to the maximum nullity problem on $\mathcal{C}_n$.}      
\label{max_nullity_c_table}
\end{center}
\end{table}
\end{theorem}
\begin{proof}
    The results given in the first row of Table \ref{max_nullity_c_table} are trivial to notice as a direct consequence of the observations stated in Theorem \ref{max_nullity_g_th}. Also, if $2 \mid n$ but $8 \nmid n$, then we can implement Theorem \ref{max_nullity_g_th} to deduce that among those graphs that maximize the nullity on the set $\mathcal{G}_n$, some are necessarily connected, such as $\mathrm{Circ}\left(n, \left\{1, \frac{n}{2} - 1 \right\}\right)$. For this reason, the maximum nullity on the set $\mathcal{C}_n$ must be the same, i.e.\ $\frac{n}{2}$, and the graphs attaining this nullity are simply those that satisfy $p + q = \frac{n}{2}$ and are connected. These observations justify the results disclosed in the second row of Table~\ref{max_nullity_c_table}. Thus, the only case that needs to be settled is when $8 \mid n$.

    Suppose that $8 \mid n$. Now, Lemma \ref{max_aux_lemma_2} tells us that each connected quartic circulant graph $G = \mathrm{Circ}(n, \{p, q\})$ satisfying $p + q = \frac{n}{2}$ necessarily has the nullity $\frac{n}{2} + 2$ and that all such graphs possess the same null space given via Eq.\ (\ref{max_aux_lemma_formula_2}). Bearing this in mind, in order to complete the proof and show the validity of the third row in Table \ref{max_nullity_c_table}, it is sufficient to demonstrate that $\eta(G) < \frac{n}{2} + 2$ whenever $p + q \neq \frac{n}{2}$. However, this is trivial to accomplish in case $v_2(p + q) \ge v_2(n)$ or $v_2(q - p) \ge v_2(n)$ by using the same technique implemented in the proof of Theorem \ref{max_nullity_g_th}.

    We will now deal with the scenario when $v_2(p + q), v_2(q - p) < v_2(n)$. By virtue of Theorem \ref{main_theorem}, we conclude that $\eta(G) \le \gcd(n, p + q) + \gcd(n, q - p)$. Since $\gcd(n, p + q) \mid \frac{n}{2}$ and $\gcd(n, q - p) \mid \frac{n}{2}$, together with $\gcd(n, p + q) \neq \frac{n}{2}$ and $\gcd(n, q - p) \neq \frac{n}{2}$, due to $p + q \neq \frac{n}{2}$, we get
    \[
        \eta(G) \le \frac{1}{2} \cdot \frac{n}{2} + \frac{1}{2} \cdot \frac{n}{2} = \frac{n}{2} ,
    \]
    which completes the proof.
\end{proof}

\newpage
\section*{Acknowledgements}

The authors would like to thank Toma\v{z} Pisanski for proposing the idea behind the problem to be solved. Furthermore, we would like to express our gratitude to the University of Primorska, University of Ljubljana and In\v{s}titut za matematiko, fiziko in mehaniko for the overall support given throughout the duration of our research.

\end{document}